\documentclass[a4paper,UKenglish]{article}
\usepackage[left=2.5cm, right=2.5cm]{geometry}

\usepackage{amsmath}
\usepackage{amsthm}
\usepackage{amssymb}
\usepackage{enumerate}
\usepackage{tikz}
\usepackage[colorlinks]{hyperref}

\definecolor{lightblue}{rgb}{0.5,0.5,1.0}
\definecolor{darkred}{rgb}{0.5,0,0}
\definecolor{darkgreen}{rgb}{0,0.5,0}
\definecolor{darkblue}{rgb}{0,0,0.5}

\hypersetup{colorlinks,linkcolor=darkred,filecolor=darkgreen,urlcolor=darkred,citecolor=darkblue}

\newtheorem{theorem}{Theorem}
\newtheorem{lemma}[theorem]{Lemma}
\newtheorem{corollary}[theorem]{Corollary}

\newtheorem{definition}[theorem]{Definition}
\theoremstyle{definition}

\theoremstyle{remark}

\newtheorem*{note*}{Note}

\newtheorem*{remark*}{Remark}

\usepackage{mathtools}
\usepackage{bbold}
\usepackage{multicol}
\usepackage{todonotes}
\title{Line Planning in Public Transport: Bypassing Line Pool Generation}

\author{
	Irene Heinrich\\
	\texttt{TU Darmstadt}
	\and
	Philine Schiewe\\
	\texttt{TU Kaiserslautern}
	\and
	Constantin Seebach\\
	\texttt{TU Kaiserslautern}
}

\date{}

\DeclareMathOperator{\cost}{cost}

\newcommand{\cfix}{c_{\text{fix}}}
\newcommand{\dfix}{d_{\text{fix}}}
\newcommand{\fmin}{f^{\min}}
\newcommand{\fmint}{\bar{f}^{\min}}
\newcommand{\fmax}{f^{\max}}
\newcommand{\femin}{f_e^{\min}}

\newcommand{\femax}{f_e^{\max}}

\newcommand{\cL}{\mathcal{L}}

\newcommand{\ftotal}{F^{(\cL,f)}}
\newcommand{\LPAL}{\textup{(LPAL)}}
\newcommand{\PMPP}{\textup{(PMPP)}}

\newcommand{\elist}{E_\textup{list}}
\newcommand{\lineends}[1]{%
	\ensuremath{\eta_{#1}}%
}
\newcommand{\feasible}{\mathcal{F}}
\usepackage{algpseudocode}
\usepackage{algorithmicx}
\usepackage{algorithm}

\usepackage{booktabs}
\usepackage{makecell}

\begin{document}

\maketitle
\thispagestyle{empty}
\begin{abstract}
Line planning, i.e.\ choosing paths which are operated by one vehicle end-to-end, is an important aspect of public transport planning. While there exists heuristic procedures for generating lines from scratch, most theoretical observations consider the problem of choosing lines from a predefined line pool. In this paper, we consider the complexity of the line planning problem when all simple paths can be used as lines. Depending on the cost structure, we show that the problem can be NP-hard even for paths and stars and that no polynomial time approximation of sub-linear performance is possible. Additionally, we identify polynomially solvable cases and present a pseudo-polynomial solution approach for trees.
\end{abstract}

\newpage
\setcounter{page}{1}

\section{Introduction}
\label{sec:introduction}
In public transport planning, \emph{lines} are important
 building blocks. As lines are (simple) paths in the public transport network that have to be covered by one vehicle end-to-end, they highly influence the subsequent steps like timetabling and vehicle scheduling, see \cite{guihaire2008transit}. On the one hand, lines influence the passengers by providing routes and transfers and on the other hand, they determine the majority of the operating costs.  Thus, line planning is an important foundation for building a public transport supply. From a set of lines, the
 \emph{line pool}, a subset of lines and their frequencies, called \emph{line concept}, is chosen for operation. While there is ample literature on line planning for a given fixed line pool, see \cite{schobel2012line}, the construction of
 line pools is often neglected. In this paper, we focus on designing line concepts without a given line pool. Instead, we consider the set of all simple paths as
 candidates thus extending the solution space. We show that depending on the cost-structure, the problem is NP-hard even for simple graph classes and that polynomial time approximations cannot give a performance guarantee that is better than linear, assuming $P\neq NP$. Additionally, we identify polynomially solvable cases and develop a pseudo-polynomial 
 algorithm
 for trees. 

%\section{Literature Review}
%\label{sec:literature}

\medskip
\noindent
\textbf{Literature review.}
%The problem of
Planning lines in public transport is extensively researched.
% in the literature.
Many %heuristic and 
(meta-)heuristic approaches exist for the
%(bus)
transit network design problem, %see
%for surveys
where lines and often also passenger routes are generated~\cite{kepaptsoglou2009transit,farahani2013review}. Usually, lines are supposed to not deviate too much from shortest paths
% as in 
\cite{arbex2015efficient,CANCELA201517}, or a set of lines to choose from is precomputed
% as in
\cite{wan2003mixed}. 

There is ample literature on line planning for a given line pool, i.e.\ a set from which lines are chosen for operation, see \cite{schobel2012line}.
% for a survey.
Most models focus either on the passengers' or the operator's perspective.
%From the passengers' point of view, 
The most important objectives
for %the 
passenger are to maximize the number of direct travelers
%, see e.g.\ 
\cite{bussieck1997optimal} or to minimize the travel time %as in
\cite{schobel2006line,Bull2018}. Here, it is especially difficult to model passenger behavior realistically, see \cite{goerigk2017line,schiewe2018line}.	
		
In this paper we focus on the operator's perspective, i.e.\ on minimizing the costs as originally introduced in \cite{claessens1998}. As in \cite{SAHIN2020102726,torres_et_al:OASIcs:2008:1583,torres2011line}, we distinguish between frequency-dependent and frequency-independent costs. Frequency-dependent costs can include costs for the distance covered by the lines and for the number of vehicles needed to operate the given line plan while frequency-independent costs can e.g.~be used to reduce the number of different lines operated.

	When solving the line planning problem for a fixed line pool, the line pool has a large influence on the complexity of the problem and the quality of solutions. One approach is to handle the generation of a suitable pool as an optimization problem itself, see \cite{GHS16}. Another possibility is to solve the line planning problem on the set of all possible lines. This idea has been studied using a column-generation approach in \cite{borndorfer2007column} where lines are only allowed to start and end at terminal stations. For this case, the line planning problem was shown to be NP-hard on planar graphs.
	%Some further
	See~\cite{borndorfer2018line,torres2011line} for further results on the complexity of the problem using terminal stations in path networks representing Istanbul Metrob{\"u}s and trees representing the Quito Troleb{\'u}s. In \cite{PaeSchiSch18}, an integer programming formulation which includes line planning on all lines is presented and applied to small instances while in \cite{masing2022price}, line planning on all \emph{circular} lines in a specific class of graphs is considered. 
	
	\medskip
	\noindent
	%In this paper,
	\textbf{Our contribution.}
	We focus on
	%the complexity of
	the line planning problem on all simple paths depending on the cost structure. We show that this problem is NP-hard %in general and
	even on planar graphs both when considering only frequency-dependent costs and when considering frequency-independent costs. 
	The inclusion of frequency-independent costs makes the problem NP-hard even on paths and stars. Considering only frequency-dependent costs, we identify both polynomially solvable and NP-hard cases. We show that the problem is also hard to solve approximately in polynomial time: A sub-linear approximation ratio would imply $P=NP$, and even with another simplification of the problem, no constant approximation ratio is possible unless $P=NP$. Additionally, we present a pseudo-polynomial
	%solution approach
	algorithm for trees and a polynomial one for %a special case.
	special cases.
	%stars.
	 An overview of these results is presented in \autoref{tab:complexity}.
	 
%The remainder of the paper is structured as follows.
\medskip
\noindent
\textbf{Outline.}
 %In \autoref{sec:literature} we discuss relevant concepts from the literature and
 In \autoref{sec:problemdef} we formally introduce the line planning on all lines problem.
 %as well as the corresponding notation.
 We present NP-hard cases in \autoref{sec:nphard} and  discuss the hardness of approximation in \autoref{sec:approximation}. \autoref{sec:poly} contains %polynomial cases 
 a polynomial algorithm for stars and in \autoref{sec:trees} we develop a pseudo-polynomial solution approach for trees as well as a polynomial version for a special case.
 %We close the paper in \autoref{sec:outlook}.
 
 \begin{table}[h]
	\begin{tabular}{p{0.14\textwidth}p{0.43\textwidth}p{0.33\textwidth}} 
		\toprule
		graph class & \makecell[l]{no frequency-independent costs\\ ($\dfix=0$)} & \makecell[l]{with frequency-independent costs\\ ($\dfix>0$)}\\ 
		\midrule 
		stars & polynomial (\autoref{thm:stars_poly}) & NP-hard (\autoref{thm:dhard_star}) \\
		\addlinespace
		paths & polynomial for $\fmax \equiv \infty$ (see \cite{MasterarbeitPhiline})
		 & NP-hard  (\autoref{thm:dhard_path})\\
		 \addlinespace
		trees & \makecell[l]{pseudo-polynomial (\autoref{thm:trees_linear})\\
		polynomial for $\fmin = \fmax$ (\autoref{thm:trees_fmin=fmax})}  & NP-hard  (Theorems \ref{thm:dhard_path} and \ref{thm:dhard_star})\\ 
		\addlinespace
		%treewidth $=2$ &pseudopolynomial (\autoref{sec:treewidth_dp}) & \\
		%treewidth $>2$ & algorithm from $=2$ can possibly be extended& \\
		planar graphs & \makecell[l]{NP-hard, even for $\{0,1\}$  input\\ (\autoref{cor:planar})} & \makecell[l]{NP-hard, even for $\{0,1\}$ input\\ (\autoref{cor:planar})}\\
		\bottomrule
	\end{tabular}
	\caption{Complexity of \LPAL{} for various graph classes.}\label{tab:complexity}
\end{table}   
	
\section{Preliminaries}
\label{sec:problemdef}

%In this section, we introduce the necessary notation and formally define the line planning on all lines problem.

\noindent
\textbf{Graph theory.}
	All graphs in this paper are finite, simple and non-empty. 
	Whenever we consider a graph $G=(V,E)$, we use $n:=|V|$ to denote its number of vertices.
	 We measure the complexity of graph problems dependent on $n$.
	The \emph{degree} $\deg(v)$ of a vertex~$v$ is the number of its neighbors.
	A graph $(V,E)$ with  $V = \{v_1, \dots, v_m\}$ and $E = \{\{v_1,v_2\}, \dots, \{v_{m-1},v_m\}\}$ where all the $v_i$ are distinct is a simple $v_1$-$v_m$-\emph{path} (or just \emph{path}). Paths can be specified as a sequence of vertices or as a sequence of edges.
	A complete bipartite graph of the form $K_{1,k}$ is a \emph{star}.

	\medskip
	\noindent
	\textbf{Line planning.}
	A \emph{public transport network (PTN)} is a graph $G=(V,E)$ whose vertices represent stations while its edges represent direct connections between the stations, e.g., streets or tracks.
	A \emph{line planning instance} is a tuple $(G, \dfix, \cfix, c,  \fmin, \fmax)$, where 
	\begin{itemize}
		\item $G=(V,E)$ is a PTN,
		\item $\dfix \in \mathbb{R}_{\geq 0}$ represents \emph{frequency-independent fixed costs},
		\item $\cfix \in \mathbb{R}_{\geq 0}$ represents \emph{frequency-dependent fixed costs},
		\item $c\colon E \to \mathbb{R}_{\geq 0}$, $e \mapsto c_e$ is a map representing the \emph{edge-dependent costs}, and
		\item $\fmin$ and $\fmax$ are \emph{integer frequency restrictions} on $E$, $e \mapsto \femin$ (respectively $e \mapsto \femax$) such that $\femin \leq \femax$ for all edges $e \in E$.
	\end{itemize}
	A \emph{line} $\ell$ is a simple path in $G$ and a \emph{line concept} $(\cL,f)$ is a set of lines $\cL$ with a \emph{frequency vector} $f=(f_{\ell})_{\ell \in \cL}\in \mathbb{N}^{|\cL|}$, i.e.\  $f_{\ell}$ is the \emph{frequency} of line $\ell$. %\philine{We also write this as multiset $\cL^f$ where for $\ell \in \cL$ the multiplicity is given by $f_{\ell}$.}
	At each edge $e \in E$, the lines sum up to a total frequency \[\ftotal_e = \sum_{\ell \in \cL \colon e \in E(\ell)} f_{\ell},\]
	where $E(\ell)$ denotes the edge set of $\ell$.
	A line concept is \emph{feasible} if for each edge $e \in E$ the frequency restrictions are satisfied, i.e.\ $\femin \leq \ftotal_e \leq \femax$.
	The set of feasible line concepts is $\feasible(G,\fmin,\fmax)$ which we may abbreviate by writing $\feasible(G)$.
	
	We use frequency-dependent \emph{line costs} $\cost_{\ell} = \cfix + \sum_{e \in E(\ell)} c_e$ which consist of fixed costs $\cfix$ and edge-dependent costs $c_e$, $e \in E$. Additionally, we use frequency-independent costs $\dfix$ per line. We define the \emph{costs of a line concept} $(\cL,f)$ as 
	\[\cost((\cL,f)) = \dfix \cdot |\cL| + \sum_{\ell \in \cL} \cost_{\ell} \cdot  f_{\ell}.\]
	
	With this notation, we can formally define the line planning on all lines problem. %considered in this paper.
	
	\begin{definition}
		Given a line planning instance, the \emph{line planning on all lines problem} \LPAL{} is to find a feasible line concept with minimal costs. 
	\end{definition}

\section{NP-hard cases}\label{sec:nphard}
For general graphs and general cost structures, the problem of finding a cost-optimal line concept is known to be NP-hard, even if
\begin{itemize}
	\item $\dfix=1$, $\cfix=0$, $c \equiv 0$,  $\femin \in \{0,1\}$ for all $e \in E$, $\fmax\equiv \infty$ or $\fmax\equiv1$ \cite{GHS16} or
	\item $\dfix=0$, $\femin \in \{0,1\}$ for all $e \in E$, $\fmax\equiv \infty$ or $\fmax \equiv1$ \cite{MasterarbeitPhiline}.
\end{itemize}

We can strengthen theses results and show that \LPAL{} is NP-hard even for subcubic planar graphs.

\begin{corollary}\label{cor:planar}
	The problem \LPAL{} is NP-hard, even if $G$ is a planar graph with maximum vertex degree at most~3 and
	\begin{itemize}
		\item $\dfix=1$, $\cfix=0$, $c \equiv 0$,  $\femin \in \{0,1\}$ for all $ e \in E$, $\fmax\equiv \infty$ or $\fmax\equiv1$ or
		\item $\dfix=0$, $\femin \in \{0,1\}$ for all $ e \in E$, $\fmax\equiv \infty$ or $\fmax \equiv1$. 
	\end{itemize}
\end{corollary}

\begin{proof}
	Ples\'{n}ik \cite{plesn1979np} shows that the Hamiltonian path problem is NP-hard even for planar digraphs with degree bound two, especially for digraphs where each vertex has either in- or out-degree one. If the reductions of \cite{GHS16,MasterarbeitPhiline} are applied to these graphs, the constructed line planning instance consists of a planar graph with vertex degree at most three. (Splitting a vertex $v$ into $v_{in}$, $v_{out}$ results in subdiving an edge into two edges.)    
\end{proof}

In the following, we show that when frequency-independent costs $\dfix$ are considered, \LPAL{} remains NP-hard even for paths and stars. 
We formulate problem reductions that utilize $\fmax$, but the following lemma can be applied to prove that hardness still holds if $\fmax\equiv\infty$, i.e.\ if no maximum constraint is put on the frequencies.

\begin{lemma}[Lifting $\fmax$]\label{lem:fmax_infty}
	Let $I=((V,E), \dfix,\cfix, c,\fmin,\fmax)$ be an instance to \LPAL{} where $\cfix=0$, $c\equiv0$ and $\fmin=\fmax$. Let $K \in \mathbb{N}$.
	Define $I':=((V,E), \dfix, \cfix, c', \fmin, \infty)$ with $c' :\equiv K+1$ and $K':= K + (K+1)\sum_{e \in E} \fmin_e$.
	
	Then $I$ has a feasible line concept with cost at most $K$ if and only if $I'$ has a feasible line concept with cost at most $K'$. Both 	
	$I'$ and $K'$ can be computed in polynomial time.
\end{lemma}
\begin{proof}
We show that we can transfer a solution $(\cL,f)$ from one instance to the other, such that it is still feasible and within the cost bound. We add a superscript to $\cost$, to distinguish for which instance we view the costs.

$I \rightarrow I'$: Clearly, $(\cL,f)$ remains feasible for $I'$. Since $c\equiv0$ and $\cfix=0$, we have $\cost^{I}((\cL,f))$ = $\dfix \cdot |\cL|$, which is less or equal to $K$.
Since $\fmin=\fmax$, the frequency-dependent line costs of $I'$ are predetermined: \begin{align*}
\sum_{\ell \in \cL} f_{\ell} \cdot \cost_{\ell}^{I'} 
&= \sum_{\ell \in \cL} f_{\ell} \cdot \left( \cfix + \sum_{e \in E(\ell)} c'_e \right) \\
&= \sum_{\ell \in \cL} f_{\ell} \sum_{e \in E(\ell)} c'_e 
= \sum_{e \in E} c'_e \sum_{\substack{\ell \in \cL: \\ e \in E(\ell)}} f_{\ell} 
= \sum_{e \in E} (K+1) \fmin_e
\end{align*} 
Then $\cost^{I'}((\cL,f)) = \dfix \cdot |\cL| + \sum_{e \in E} (K+1) \fmin_e \leq K + \sum_{e \in E} (K+1) \fmin_e = K'$.

$I' \rightarrow I$: Towards a contradiction, assume $\fmin_e+1 \leq \sum_{\ell \in \cL \colon e \in E(\ell)} f_{\ell}$ for some $e \in E$. Then we can derive in a similar fashion: 
\begin{align*}
\cost^{I'}((\cL,f)) &= \dfix \cdot |\cL| + \sum_{e \in E} (K+1) \sum_{\substack{\ell \in \cL: \\ e \in E(\ell)}} f_{\ell}
\geq \dfix  \cdot |\cL| + (K+1) + \sum_{e \in E} (K+1) \fmin_e \\
& > K + (K+1) \sum_{e \in E} \fmin_e = K'
\end{align*} 
\begin{sloppypar}
This contradicts our assumption $\cost^{I'}((\cL,f)) \leq K'$. Therefore we have \linebreak $\fmin_e = \sum_{\ell \in \cL \colon e \in E(\ell)} f_{\ell}$ for all $e \in E$, implying that $(\cL,f)$ is a feasible line concept for $I$.
Subtracting the now fixed frequency-dependent line costs, we see that $\cost^{I}((\cL,f)) = \dfix \cdot |\cL| \leq K$.
\end{sloppypar}
\end{proof}

First, we show that \LPAL{} is NP-hard on paths.

\begin{theorem}\label{thm:dhard_path}
	The problem \LPAL{} is NP-hard, even if $G$ is a path and $\fmin= \fmax$ or $\fmax\equiv\infty$.
\end{theorem}
\begin{proof}
We show a reduction of the 3-Partition problem \cite{DBLP:books/fm/GareyJ79} to the decision version of \LPAL{}, first for the case $\fmin = \fmax$. 
Let a multiset of positive integers $S=\{x_1,\dots,x_{3p}\}$ be given, in an arbitrary order. 
The idea of our construction is to have a path with one interval of monotonically increasing frequency constraints, and another interval with monotonically decreasing frequency constraints. The first interval represents partitions $S_1, \ldots, S_p$ while the second interval represents the elements of $S$.
By choosing the frequency restrictions, we force the multiset of line frequencies to be exactly $S$.
Then we can construct lines to have one end in the first interval, and the other end in the second interval, representing to which set $S_k$ an element $x_i \in S$  is assigned.
In the first interval, lines can overlap in different ways, each representing a different way to partition $S$.

Define $h:=\sum {S} / p$. %We may assume $m$ to be divisible by 3, and $h$ to be an integer, as otherwise the instance would be trivial.
We may assume $h$ to be an integer.
Additionally we can assume that every subset of $S$ which sums to $h$, contains exactly 3 elements; this does not weaken the 3-Partition problem.

Now define a sequence of integers, used for constructing the frequency  restrictions: 
\[a_i := \begin{cases}
h &\text{if~}i \leq 0\\
-x_{i} &\text{if~}i > 0
\end{cases} \quad\text{ for } i \in [1-p, 3p]. \text{ (Note that indices may be negative.)} \]
We construct our instance $I=(G, \dfix,\cfix, c,\fmin,\fmax)$ with decision parameter $K$ as follows:
\begin{multicols}{4}
	\begin{itemize}
		\item $\dfix := 1$
		\item $\cfix := 0$
		\item $c :\equiv 0$
		\item $K := 3p$
	\end{itemize}
\end{multicols}
The graph $G$ is a path on $4p$ vertices, which we call $v_{1-p}, \dots, v_{3p}$. The edges are $e_i := \{v_i, v_{i+1}\}$ for $i \in [1-p, 3p-1]$.
For all $i \in [1-p, 3p-1]$, we set $\fmin_{e_i} := \fmax_{e_i} := \sum_{j=1-p}^{i} a_j$. The construction is illustrated in \autoref{fig:dhard_path:example}.

Consider a feasible solution $(\cL,f)$ for $I$ with $\cost((\cL,f)) \leq K$.
From $\dfix=1$ follows that $|\cL| \leq 3p$.
Since $G$ is a path, we can say that every line of $(\cL,f)$ has a left and a right end. We first argue the case where the left end of each line $(v_i, \ldots, v_j)$ is in the first interval, i.e., $v_i$ satisfies $i \in [1-p,0]$, and the right end is in the second interval, i.e., $v_j$ satisfies $j \in [1,3p]$.

For any $i\in [0,3p-2]$ we have $\ftotal_{e_{i}} > \ftotal_{e_{i+1}}$, implying that at least one line must have a right end at $v_{i+1}$. Also some line must end at $v_{3p}$, since $\ftotal_{e_{3p-1}} = ph - \sum_{j=1}^{3p-1} x_i = x_{3p} > 0$. 
Hence $\cL$ consists of exactly $3p$ lines, each having the right end at a different $v_i$ for $i\in [1, 3p]$.

Define $\ell_i$ to be the unique line ending at $v_i$. To make up the frequency difference $x_i$ in the graph $G$, we must have $f_{\ell_i} = x_i$. 
%Hence for every $x\in S$ we have a line with frequency $x$. 
Now, for any $j\in [1-p,0]$, consider the subset $\hat L_j$ of lines which have their left end at $v_j$. Their frequencies must sum up to $\ftotal_{e_{j}} - \ftotal_{e_{j-1}} = a_j = h$.
%Since $\ftotal_{e_{0}} = \sum{S} = \sum_{\ell \in \cL} f_{\ell}$, all lines cover $e_{0}$ and hence have their left ends at some $v_j$ with $j\in [1-p,0]$.
Since all lines  have their left ends at some $v_j$ with $j\in [1-p,0]$, it follows that the sets $\hat L_{1-p}, \dots, \hat L_0$ partition $\cL$, and correspond to a partition of $S$, where each subset has sum $h$. This solves the 3-Partition problem.

If there is a line whose left and right end are in the second interval, it is no longer guaranteed that $f_{\ell_i} = x_i$ for the line $\ell_i$ with right end at $v_i$. Instead, $f_{\ell_i}$ is increased by the total frequency of lines whose left ends are at $v_i$. Now, we can elongate all lines with left end at $v_i$ to the left end of $\ell_i$ and reduce the frequency of $\ell_i$ to $x_i$ without introducing new lines or changing the total frequency of any edge. As there are no lines whose left end is $v_{3p}$ and the right end of $\ell_1$ has to be in the first interval, this allow us to construct a solution in the desired form in linear time. 

For the other direction, consider a given solution to the 3-Partition problem: $S_1 \cup \dots \cup S_{p} = S$ with $\sum{S_k} = h$ for all $k$. We construct a feasible line concept as follows: For every $i\in [1,3p]$, create a line $\ell_i$ with frequency $x_i$, having its right end at $v_i$. If $x_i \in S_k$ then $\ell_i$ has its left end at $v_{1-k}$. It is easy to check that these lines sum up exactly to the frequency profile of $G$, and the cost $K$ is not exceeded.

To show hardness for the case $\fmax\equiv\infty$, we can apply \autoref{lem:fmax_infty}.
\end{proof}

\begin{figure}
\begin{center}
\begin{tikzpicture}[xscale=1.5,yscale=1.2]
\tikzset{
	knoten/.style={draw,rounded corners,minimum size=5pt}
}
\node[knoten] (va) at (-2,0) {$v_{-1}$};
\node[knoten] (vb) at (-1,0) {$v_{0}$};
\node[knoten] (v1) at (0,0) {$v_{1}$};
\node[knoten] (v2) at (1,0) {$v_{2}$};
\node[knoten] (v3) at (2,0) {$v_{3}$};
\node[knoten] (v4) at (3,0) {$v_{4}$};
\node[knoten] (v5) at (4,0) {$v_{5}$};
\node[knoten] (v6) at (5,0) {$v_{6}$};

\draw (va) -- (vb) node[midway,above] {$10$};
\draw (vb) -- (v1) node[midway,above] {$20$};
\draw (v1) -- (v2) node[midway,above] {$19$};
\draw (v2) -- (v3) node[midway,above] {$17$};
\draw (v3) -- (v4) node[midway,above] {$15$};
\draw (v4) -- (v5) node[midway,above] {$11$};
\draw (v5) -- (v6) node[midway,above] {$6$};

\draw [|-|,thick,blue](-2,1.4)  -- (0,1.4) node[right]{1};
\draw [|-|,thick,blue](-1,1.2)  -- (1,1.2) node[right]{2};
\draw [|-|,thick,blue](-1,1.0)  -- (2,1.0) node[right]{2};
\draw [|-|,thick,blue](-2,0.8)  -- (3,0.8) node[right]{4};
\draw [|-|,thick,blue](-2,0.6)  -- (4,0.6) node[right]{5};
\draw [|-|,thick,blue](-1,0.4)  -- (5,0.4) node[right]{6};
\end{tikzpicture}
\end{center}
\caption{Example for the construction from \autoref{thm:dhard_path}, along with feasible line concept. Here $S=\{1,2,2,4,5,6\}$.}\label{fig:dhard_path:example}
\end{figure}
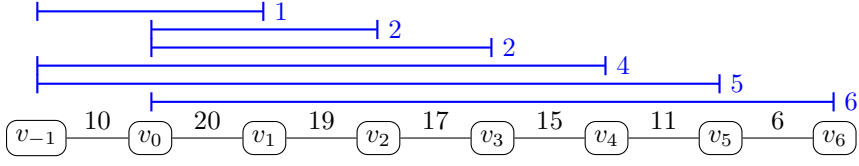

Additionally, \LPAL{} is NP-hard on stars.

\begin{theorem}\label{thm:dhard_star}
	The problem \LPAL{} is NP-hard, even if $G$ is a star and $\min = \fmax$ or $\fmax\equiv\infty$.
\end{theorem}
To prove \autoref{thm:dhard_star}, we translate \LPAL{} with the assumptions of \autoref{thm:dhard_star} into a more abstract combinatorial problem:
\begin{definition}
	The \emph{Partition into many partitions problem} \PMPP{} is the following decision problem:
	
	Input: A set of positive integers $S$ and a number $K$. 
	
	Question: Can a subset $S'\subseteq S$ be partitioned into at least $K$ nonempty sets, such that each in turn is a yes-instance to the Partition problem?
\end{definition}
\begin{lemma}\label{lem:hardness_pmpp}
	\PMPP{} is NP-hard.
\end{lemma}
\begin{proof}\label{proof:hardness_pmpp}
	This proof is an adaptation of the reduction described in \cite{HULETT2008594}.
	We reduce Partial Latin Square Completion (PLSC) to \PMPP{}. For a definition of PLSC see \cite{HULETT2008594}.
	
	Consider a partial Latin square $L$ of dimension $p \times p$ with $m$ missing entries. 
	Define $q:=6p-2$. 
	We construct a \PMPP{} instance from the Latin square by defining $K:=m$ and putting the following numbers into the set $S$:
	\begin{itemize}
		\item If color $c$ does not occur in row $k$, put $x(k,c):= q(2k-1)-(2c-1)$ into $S$.
		\item If color $c$ does not occur in column $\ell$, put $y(\ell,c):= q^2(2\ell-1)+(2c-1)$ into $S$.
		\item If the cell in row $k$ and column $\ell$ is empty, put $z(k,\ell):= q^2(2\ell-1)+q(2k-1)$ into $S$.
	\end{itemize}
	\PMPP{} requires that $S$ only contains positive numbers. A quick check of the x-numbers shows that they are positive: $x(k,c) \geq q-(2c-1) \geq q-(2p-1) = 4p-1>0$. The y- and z-numbers are positive since $\ell$, $c$ and $k$ each are positive.
	
	We check that these numbers indeed form a set of size $3m$, i.e.\ they are pairwise different:
	Assume $x(k_1,c_1) = x(k_2,c_2)$ holds for some $k_1,c_1,k_2,c_2 \in [1,p]$. Considering this equation modulo $q$, we find $(2c_1-1) \equiv (2c_2-1) \mod q$. Since $q>2p-1$, it follows: $c_1=c_2$. Hence the equation simplifies to $q(2k_1-1)=q(2k_2-1)$, so also $k_1=k_2$. This shows that the x-numbers are created by an injective map. The same arguments work for pairs of y-numbers and pairs of z-numbers.
	Now assume $x(k_1,c_1) = y(\ell_2,c_2)$. It follows: $(2c_1-1)+(2c_2-1) \equiv 0 \mod q$. This is a contradiction, since $4p-2<q$.
	Assume $x(k_1,c_1) = z(k_2,\ell_2)$ or $y(\ell_1,c_1) = z(k_2,\ell_2)$. In both cases $(2c_1-1) \equiv 0 \mod q$ would follow, which is a contradiction.
	
	Now we consider all the ways 3 or fewer of these numbers can be a yes-instance to the Partition problem.
	A single number cannot be a yes-instance.
	Two numbers also cannot be a yes-instance, since $S$ is a set and every number is different.
	Here we work out only some of the possible three-number combinations. The rest can be calculated similarly.
	\begin{itemize}
		\item $z(k_1,\ell_1) = x(k_2,c_2)+y(\ell_3,c_3)$. Considering this equation modulo $q$, we find that $c_2=c_3$. Then, dividing by $q$ and again applying modulo, we get $k_1=k_2$ and finally $\ell_1=\ell_3$.
		\item $z(k_1,\ell_1)+x(k_2,c_2) = y(\ell_3,c_3)$. It would follow: $(2c_2-1)+(2c_3-1) \equiv 0 \mod q$, which is not possible, as we have seen before.
		%\item $x(k_2,c_2) = y(\ell_3,c_3)+z(k_1,\ell_1)$. Same as above.
		\item $x(k_1,c_1) = y(\ell_2,c_2)+y(\ell_3,c_3)$. It would follow: $(2c_1-1)+(2c_2-1)+(2c_3-1) \equiv 0 \mod q$. This is not possible, since $0<(2c_1-1)+(2c_2-1)+(2c_3-1) \leq 6p-3<q$.
		%\item $y=x+x$. Same as above.
		%\item $x(k_1,c_1) = z(k_2,\ell_2)+z(k_3,\ell_3)$. It would follow: $(2c_1-1) \equiv 0 \mod q$, a contradiction.
		%\item $y=z+z$. Same as above.
		%\item $z=x+x$ or $z=y+y$. Implies $(2c_1-1)+(2c_2-1) \equiv 0 \mod q$, not possible.
		\item $x(k_1,c_1) = x(k_2,c_2)+x(k_3,c_3)$. Consider this equation modulo 2. Since $q$ is even, we would obtain $-1\equiv-2 \mod 2$, which is a contradiction. The case of three y-numbers is dealt with in the same way. In the case of three z-numbers, first divide by $q$.
		%\item $x=x+y$. Same as $x+x=x$.
		%\item $x=x+z$. Same as $x=x$.
		%\item $y=y+x$. Same as $y+y=y$.
		%\item $y=y+z$. Same as $y=y$.
		%\item $z=z+x$ or $z=z+y$. Same as $z+z=x$.
	\end{itemize}
	After considering all combinations, we find that the only way 3 numbers can be a yes-instance to Partition, is by choosing one number from each family x,y and z; importantly these numbers must have matching choices for row, column and color.
	
	Now let $B_1,\dots,B_m$ be a solution to \PMPP{}, i.e.\ the $B_i$ are nonempty yes-instances to Partition, are pairwise disjoint and their union is a subset of $S$.
	As we have shown, each $B_i$ contains at least three elements. Since $|S|=3m$, every element of $S$ is used and no $B_i$ can contain more than three elements. Then each $B_i$ corresponds to a triple of x-,y- and z-numbers, which in turn corresponds to a row $k$, a column $\ell$ and a color $c$. We then fill our partial Latin square, by coloring the cell at row $k$ and column $\ell$ with $c$, repeating this for every $B_i$.
	Since every z-number was used, the Latin square must now be filled. It is also a valid coloring, since for every row/column each missing color appears only in one x-number/y-number.
	
	For the other direction, consider a valid completion of the partial Latin square. Then for each of the $m$ new colorings $c_i$ in the cell at row $k_i$ and column $\ell_i$, we create $B_i:=\{z(k_i,\ell_i), x(k_i,c_i), y(\ell_i,c_i)\}$. Then each $B_i$ is a yes-instance to the Partition problem, and is contained in $S$. The created sets are pairwise disjoint, since the Latin square would otherwise have a collision. 
\end{proof}

Before giving a detailed proof of \autoref{thm:dhard_star}, we sketch the ideas behind it:

\begin{figure}
	\begin{center}
		\begin{tikzpicture}[scale=1.1]
			\tikzset{
				knoten/.style={draw,rounded corners,minimum size=5pt}
			}
			\node[knoten] (c) at (0,0) {$c$};
			\node[knoten] (v1) at (0:1) {$v_{1}$};
			\node[knoten] (v2) at (90:1) {$v_{2}$};
			\node[knoten] (v3) at (180:1) {$v_{3}$};
			\node[knoten] (v4) at (270:1) {$v_{4}$};
			\node (lb) at (-1,1) {$G$};
			
			\draw (c) -- (v1) node[draw=none,fill=none,font=\scriptsize,midway,above] {$5$};
			\draw (c) -- (v2) node[draw=none,fill=none,font=\scriptsize,midway,left] {$3$};
			\draw (c) -- (v3) node[draw=none,fill=none,font=\scriptsize,midway,above] {$4$};
			\draw (c) -- (v4) node[draw=none,fill=none,font=\scriptsize,midway,left] {$2$};

			\begin{scope}[shift={(3,0)}]
				\node (lb) at (-1,1) {$(\cL,f)$};
				\draw [|-|,thick,blue] (1,0+0.1) .. controls (0,0.1) .. (0,1+0.1) node[midway,above]{3};
				\draw [|-|,thick,cyan] (0:1) -- (180:1) node[midway,below,xshift=0.1cm]{2};
				\draw [|-|,thick,red] (0,-1-0.1) .. controls (0,-0.1) ..  (-1,0-0.1) node[midway,below]{2};
			\end{scope}
			
			\node (arr) at (3+1.75,0) {$\rightarrow$};
			
			\begin{scope}[shift={(6.5,0)}]
				\tikzset{
					knoten/.style={draw,rounded corners,minimum size=5pt}
				}
				\node[knoten,fill=magenta!50] (w1) at (0:1) {$w_{1}$};
				\node[knoten,fill=yellow!50] (w2) at (90:1) {$w_{2}$};
				\node[knoten,fill=yellow!50] (w3) at (180:1) {$w_{3}$};
				\node[knoten,fill=magenta!50] (w4) at (270:1) {$w_{4}$};
				\node (lb) at (-1,1) {$H_{\cL}$};
				
				\draw[blue] (w1) -- (w2) node[midway,above] {$3$};
				\draw[cyan] (w1) -- (w3) node[midway,above] {$2$};
				\draw[red] (w3) -- (w4) node[midway,above] {$2$};
				
				\node (arr) at (3,0) {$\implies 3+4=2+5$};
			\end{scope}
		\end{tikzpicture}	
	\end{center}
	\caption{Example of the relationship between line concepts on stars and number partitions.}\label{fig:dhard_star:example}
\end{figure}
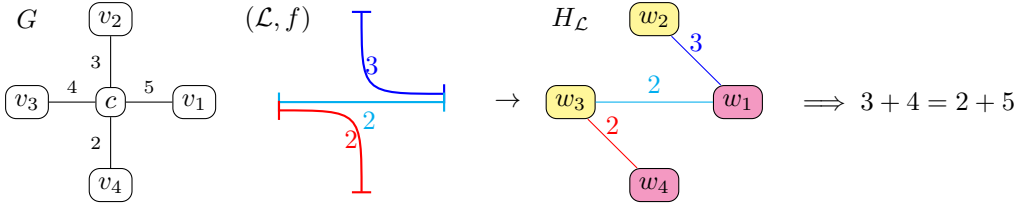

Consider an optimal feasible line concept $(\cL,f)$ on a star graph $G$ where $\fmin=\fmax$ and the cost only includes the number of lines. We observe that whenever a one-edge line $\ell_1\in \cL$ shares an edge with some other line $\ell_2 \in \cL$, we may obtain an equivalent line concept $\cL'$ without edge-sharing, by shortening $\ell_2$ and increasing the frequency of $\ell_1$. 
Hence, we may assume that only two-edge lines can share an edge with one another. We can visualize the edge intersection between lines as a graph $H_{\cL}$ where each two-edge line $(v_i,c,v_j)$ is represented by an edge $\{v_i,v_j\}$. 
The resulting graph $H_{\cL}$ cannot have a cycle, because otherwise we could shift around frequencies to remove some line, yielding a better line concept.
Hence $H_{\cL}$ is a forest, and we can use each contained tree, as shown in \autoref{fig:dhard_star:example}, to obtain a number partition on some subset of the edge frequencies.
This makes the equivalence to \PMPP{} apparent.

With this intuition, the proof of \autoref{thm:dhard_star} proceeds as follows:
\begin{proof}\label{proof:dhard_star}
	We show a reduction of \PMPP{} to the decision version of \LPAL{}.
	Let the set of positive integers $S=\{x_1,\dots,x_m\}$ and the lower bound $K$ be given as an input to \PMPP{}. 
	
	We construct our instance $I=(G=(V,E), \dfix,\cfix, c,\fmin,\fmax)$ with decision parameter $K'$ as follows:
	\begin{multicols}{4}
		\begin{itemize}
			\item $\dfix := 1$
			\item $\cfix := 0$
			\item $c :\equiv 0$
			\item $K' := m-K$
		\end{itemize}
	\end{multicols}
	For the construction of $G$ and $\fmin=\fmax$ see \autoref{fig:dhard_star}.
	\begin{figure}
		\centering
		\begin{tikzpicture}[scale=2]
			\tikzset{
				knoten/.style={draw,rounded corners,minimum size=5pt}
			}
			\node[knoten] (c) at (0,0) {$c$};
			\node[knoten] (v1) at (-1,0) {$v_{1}$};
			\node[knoten] (v2) at (-0.707,0.707) {$v_{2}$};
			\node[knoten] (v3) at (0,1) {$v_{3}$};
			\node[knoten] (vn) at (1,0) {$v_{m}$};
			
			\node (dots) at (0.5,0.5) {$\ddots$};
			
			\draw (c) -- (v1) node[draw=none,fill=none,font=\scriptsize,midway,above] {$x_1$};
			\draw (c) -- (v2) node[draw=none,fill=none,font=\scriptsize,midway,above] {$x_2$};
			\draw (c) -- (v3) node[draw=none,fill=none,font=\scriptsize,midway,right] {$x_3$};
			\draw (c) -- (vn) node[draw=none,fill=none,font=\scriptsize,midway,above] {$x_m$};

			%\draw (v2) -- (dots)  {};
			%\draw (dots) -- (vnn)  {};
			
		\end{tikzpicture}
		\caption{Line planning instance constructed in \autoref{thm:dhard_star}.}\label{fig:dhard_star}
	\end{figure}
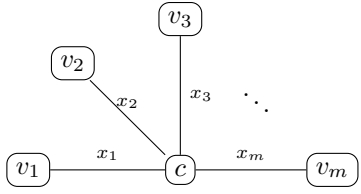
	
	Let a solution $Sol=((A_1,B_1),\dots,(A_K,B_K))$ to the \PMPP{} instance be given, i.e.\ the sets $A_1,\dots,A_K,B_1,\dots,B_K$ are nonempty and form a partition of some subset of $S$, with $\sum A_i = \sum B_i$ for all $i\in[1,K]$. 
	%We construct a solution to \LPAL{} using the following algorithm:
	We construct a solution to \LPAL{} using \autoref{algo:PMPP2LPAL}.
	
	\begin{algorithm}[h]
		\caption{Constructing a line concept from a \PMPP{} solution} \label{algo:PMPP2LPAL} 
		\begin{algorithmic}[1]
			\State $\cL=\emptyset$
			\For{$(A,B) \in Sol$}
			\State Treat the numbers in $A$ and $B$ as mutable data structures, which can store a value and an index.
			\State Assign to each $y$ in $A$ and $B$ an index $i$ such that $x_i = y$.
			\While{$|A|>0$ and $|B|>0$}
			\State $a = \min(A)$
			\State $b = \min(B)$
			\If{$a < b$}
			\State Add to $\cL$ a line from $v_{a.index}$ to $v_{b.index}$ with frequency $a$
			\State A.remove(a)
			\State b -= a
			\ElsIf{$a > b$}
			\State Add to $\cL$ a line from $v_{a.index}$ to $v_{b.index}$ with frequency $b$
			\State B.remove(b)
			\State a -= b
			\Else
			\State Add to $\cL$ a line from $v_{a.index}$ to $v_{b.index}$ with frequency $a$
			\State A.remove(a)
			\State B.remove(b)
			\EndIf
			\EndWhile
			\EndFor
			\For{$i\in[1,n]$ where $v_i$ is not covered yet}
			\State Add to $\cL$ a line from $v_{i}$ to $c$ with frequency $x_i$
			\EndFor
		\end{algorithmic}
	\end{algorithm}
	
	This algorithm uses each partition to reduce the number of lines needed to create a feasible line concept. We observe: %the following properties of the algorithm:
	\begin{itemize}
		\item When entering line 2, it holds: Every edge of $G$ which is covered by some lines of $\cL$, has the lines add up to the correct frequency.
		\item When entering line 5, it holds: $\Sigma A = \Sigma B$. 
		\item Every iteration of the while-loop adds one new line, and removes at least one entry from $A$ or $B$.
		\item Since $\Sigma A = \Sigma B$ is an invariant, the case of line 17 must be reached eventually for every $(A,B) \in Sol$.
	\end{itemize}
	Define $S':= \bigcup_{i=1}^{K} A_i \cup B_i$. It follows: 
	\begin{itemize}
		\item In each iteration of the outer for-loop, at most $|A|+|B|-1$ lines are created.
		\item Before reaching line 23, at most $\sum_{i=1}^{K} (|A_i|+|B_i|-1) = |S'|-K$ lines are created.
		\item In the remaining algorithm, exactly $|S\setminus S'|$ lines are created.
	\end{itemize}
	It is easy to check that the algorithm creates a feasible line concept. It contains at most $|S'|-K + |S\setminus S'| = |S|-K = m-K = K'$ lines, so its cost is below the bound.

	For the other direction, consider a solution $(\cL,f)$ to $I$ with cost at most $K'$, i.e.\ with at most $m-K$ lines.
	Let $m_1$ be the number of one-edge lines, and $m_2$ the number of two-edge lines in $\cL$.
	We may assume that one-edge lines do not overlap with any other line, as in that case we can cut down the other lines and adjust the frequency, preserving the line concept's feasibility.
	We construct an auxiliary graph $H_{\cL}$ as follows: Start with the vertices $v_1,\dots,v_m$. 
	Any two-edge line $l\in \cL$ connects some $v_i$ to some $v_j$, with $i\neq j$. For each such line, create an edge $\{v_i,v_j\}$ in $H_{\cL}$.
	Any one-edge line $l\in \cL$ connects $c$ to some $v_i$. For each such line, delete the vertex $v_i$. We do not need to delete any edges, since we assumed that one-edge lines do not overlap.
	Now $H_{\cL}$ has $m-m_1$ vertices and $m_2 = K'-m_1 = m-m_1-K$ edges. Hence $H_{\cL}$ contains a forest of at least $K$ separate trees. Call these trees $T_1,\dots,T_K$.
	Since every edge of $G$ has a positive $\fmin$, every edge must be covered by some line $l\in \cL$. It follows that every vertex of $H_{\cL}$ has at least one incident edge. This means that each tree $T_i$ contains at least two vertices.
	
	Each tree $T_i$ is also a bipartite graph. Let $(P_1:=\{v_j \colon j\in J_1\}, P_2:=\{v_j \colon j\in J_2\})$ for some $J_1,J_2$ be a bipartition of $T_i$. By construction, every line $\ell\in\cL$ is either disjoint from $T_i=P_1 \cup P_2$, or connects $P_1$ with $P_2$, i.e.\ $\emptyset \neq V(\ell) \cap P_1 \iff \emptyset \neq V(\ell) \cap P_2$. Therefore it holds:
	\begin{align*}
		\sum_{j\in J_1} x_j &= \sum_{j\in J_1} \sum_{\substack{\ell\in\cL:\\ \{v_j, c\}\in E(\ell)} } f_\ell
		= \sum_{\substack{\ell\in\cL:\\ \emptyset \neq V(\ell) \cap P_1 }} f_\ell \\
		&= \sum_{\substack{\ell\in\cL:\\ \emptyset \neq V(\ell) \cap P_2 }} f_\ell 
		= \sum_{j\in J_2} \sum_{\substack{\ell\in\cL:\\ \{v_j, c\}\in E(\ell) }} f_\ell
		= \sum_{j\in J_2} x_j .
	\end{align*}
	Then $(\{ x_j \colon j\in J_1 \}, \{ x_j \colon j\in J_2 \})$ is a nonempty solution to the Partition problem. We repeat this for every tree $T_i$, to get $K$ disjoint number partitions, solving \PMPP{}.
	
	To show hardness for the case $\fmax\equiv\infty$, we can apply \autoref{lem:fmax_infty}.
\end{proof}

The presented hardness results in this section actually show \emph{strong} NP-hardness, i.e.\ even when we restrict the numerical parameters of \LPAL{} instances to be (polynomially) small compared to the graph, the problem remains NP-hard.

\section{Hardness of approximation}\label{sec:approximation}
In this section we show that in the case $\dfix=0$, no polynomial time approximation algorithm for \LPAL{} can have a sub-linear performance ratio.
Even when additionally $\fmax\equiv\infty$, no constant-factor polynomial time approximation is possible.
\begin{theorem}
	Assuming $P\neq NP$, the problem \LPAL{} cannot be approximated within a factor of $n^{1-\epsilon}$ by a polynomial-time algorithm, even in the case $\dfix=0$.
\end{theorem}
\begin{proof}
We prove this using a gap-producing reduction from the Hamiltonian path  problem (for a definition, see \cite{plesn1979np}).

Let $G$ be a directed graph of size $n$. We claim that a line planning instance \linebreak $I=(G', \dfix, \cfix, c, \fmin, \fmax)$ with $\dfix=0$, $\cfix=1$ and $c\equiv0$ can be constructed from $G$ in polynomial time, where two special vertices $v_1$ and $v_2$ are marked, and it holds:
\begin{itemize}
	\item If $G$ has a Hamiltonian path: $I$ can be solved using a single line, with endings $v_1$ and $v_2$.
	\item If $G$ has no Hamiltonian path: $I$ cannot be solved using a single line, i.e.\ $I$ requires at least two lines.
\end{itemize}
To construct $I$, we first apply the reduction from \cite{GHS16} to $G$, creating an \LPAL{} instance that can be solved using a single line if and only if $G$ has a Hamiltonian path. Then we add two vertices $u_1$ and $u_2$, which we connect to all other vertices, using edges having $\femin=0$. To $u_1$ we connect a new vertex $v_1$, likewise we connect a new vertex $v_2$ to $u_2$, this time using edges having $\femin=1$. This forces a line from $v_1$ to $v_2$, but otherwise preserves the reduction equivalence.

Starting from $G'$, for any $k \in \mathbb{N}_{\geq 1}$, we can construct a graph $G_k$ as follows: Create $k$ copies of $G'$, and for all $i\in [1,k-1]$, add an edge between $v_2$ of copy $i$ and $v_1$ of copy $i+1$. Call these $k-1$ new edges \emph{connectors}. Then $G_k$ consists of $k n$ vertices.
Consider a new line planning instance $I_k$ on $G_k$, where we also copied the weights from $I$ onto $G_k$, and have $\femax=1$ on the connectors.
If $G$ has a Hamiltonian path, then $I_k$ can be solved using a single line, which we get by concatenating the lines for each copy of $G'$, with help of the connectors.

We say a line $\ell$ \emph{visits} copy $i$, if the vertices of $\ell$ and the $i$-th copy of $G'$ intersect. If $G$ has no Hamiltonian path, then each copy of $G'$ must be visited by at least two different lines, totaling at least $2k$ visits. A single line can visit multiple copies of $G'$, but must cross a connector for each additional visit. Since $\femax=1$, every connector can only be crossed once. This affords us $k-1$ visits. The remaining $k+1$ visits must be paid by different lines, i.e.\ any line concept $\cL$ solving $I_k$ needs $|\cL|\geq k+1$.

Now assume that for some $\epsilon \in (0,1]$, we can approximate \LPAL{} within $n^{1-\epsilon}$ using an algorithm $A$  in polynomial time.
Then we define $k:= \lfloor 1+ n^{(1-\epsilon)/\epsilon} \rfloor$, i.e.\ $k$ is the smallest integer larger than $n^{(1-\epsilon)/\epsilon}$.
Given a graph $G$ as input to the Hamiltonian path problem, we construct $I_k$, which has size $kn$, which is bounded by a polynomial in $n$. Then apply algorithm $A$, to get an approximate solution of cost $a$.
If $G$ has a Hamiltonian path, then the optimal solution to $I_k$ has value $1$. Hence $a \leq 1\cdot (kn)^{1-\epsilon}$ and 
$a / k \leq n^{1-\epsilon} k^{-\epsilon} < n^{1-\epsilon} \cdot (n^{(1-\epsilon)/\epsilon})^{-\epsilon}
= 1$. Thus $a < k$. \\
If $G$ has no Hamiltonian path, then the optimal solution to $I_k$ has value at least $k$, hence also $a\geq k$.
By comparing $a$ to $k$, we can determine whether $G$ has a Hamiltonian path in polynomial time, implying $P=NP$.
\end{proof}

%Note that this result is tight in the sense that an $n$-approximation for \LPAL{} is possible in polynomial time: Given $I=(G=(V,E), \dfix, \cfix, c, \fmin, \fmax)$, simply create a single-edge line for every edge that needs to be covered, i.e.\ every $e\in E':=\{ e\in E \mid \femin>0 \}$. The cost for this is \[c_a := \sum_{e\in E'} (\dfix + \femin\cdot(\cfix+c_e)) .\]
%Since any line can cover at most $|V|-1$ edges, even an optimal solution requires at least $|E'|/(|V|-1)$ lines. This results in an optimal cost of at least \[c_o := |E'|/(|V|-1)\cdot\dfix + \sum_{e\in E'} \femin \cdot \cfix/(|V|-1) + \sum_{e\in E'} \femin\cdot c_e ,\]
%yielding an approximation ratio of
%\begin{align*}
%c_a/c_o &= \frac{\sum_{e\in E'} \dfix + \sum_{e\in E'} \femin\cfix} {c_o} + \frac{\sum_{e\in E'} \femin\cdot c_e} {c_o} \\
%&\leq \frac{\sum_{e\in E'} \dfix + \sum_{e\in E'} \femin\cfix} {|E'|/(|V|-1)\cdot\dfix + \sum_{e\in E'} \femin \cdot \cfix/(|V|-1)} + \frac{\sum_{e\in E'} \femin\cdot c_e} {\sum_{e\in E'} \femin\cdot c_e} \\
%&\leq |V|-1 + 1 = n .
%\end{align*}
%This of course is a terrible approximation. To get better results we need to put more restrictions on the instances we consider. In the preceding hardness proof, it was essential that we can use $\fmax$ to bound the frequencies.

Since an $n$-approximation (or worse) for \LPAL{} is useless in practice, we want to weaken the lower bound by putting more restrictions on the considered instances. In the preceding hardness proof, it was essential that we can use $\fmax$ to bound the frequencies.

In contrast we now consider instances, where $\fmax\equiv\infty$.
\begin{theorem} \label{thm:approx2}
	Assuming $P\neq NP$, the problem \LPAL{} cannot be approximated within any constant factor by a polynomial-time algorithm, even in the case $\dfix=0$ and $\fmax\equiv\infty$.
\end{theorem}
\begin{proof}\label{proof:thm:approx2}
	In this proof we assume all \LPAL{} instances \linebreak $I=(G=(V,E), \dfix, \cfix, c, \fmin, \fmax)$ to have $\fmax\equiv\infty$, $\dfix=0$, $\cfix=1$, $c\equiv0$ and $\fmin_e \in \{0,1\}$ for all $e \in E$.
	
	Let $I=(G=(V,E), \dfix, \cfix, c, \fmin, \fmax)$ be an \LPAL{} instance. We call any edge $e=\{v_1,v_2\} \in E$ with $\fmin_e = 1$ and $\deg(v_1)=1$ an \emph{antenna}. Declare $v_1$ to be the \emph{tip} of the antenna.
	We call $I$ \emph{nice} if it has exactly two distinct antennae.
	
	Let $p$ be a path on $G$ and $V' \subseteq V$. The \emph{restriction} of $p$ to $V'$ is the sequence obtained from $p$ by removing all vertices not in $V'$. We call the restriction \emph{proper} if it is a path on $V'$.
	
	When we add a new antenna to an instance $I$, it may need more lines to be solved, but certainly not fewer. If $I$ can be solved using just a single line, then it has at most two antennae. If $I$ has exactly two antennae, then the single line must have its ends at the antenna tips.
	However, if $I$ has fewer than two antennae, there is a way to attach new antennae until we have two, such that the resulting instance can still be solved using a single line.
	
	Let $I$ and $J$ be nice instances.
	Define $I\times J$ as the following construction:
	Take $J$ and replace every edge $e=\{u,v\}$, where $\femin=1$, by a copy of $I$; then $u$ and $v$ are identified with the antenna tips of that copy.
	For an example of $I\times J$, see \autoref{fig:ItimesJ}.
	
	We claim that $I\times J$ is also nice. In particular the antennae of $I\times J$ are part of two different copies of $I$. Denote these copies by $A^1_{I\times J}$ and $A^2_{I\times J}$.
	Let $\ell$ be a path on $I\times J$ and $C$ be some copy of $I$ which is part of $I\times J$. There are only two vertices where $\ell$ can enter or leave $C$. If $\ell$ starts outside $C$, it can enter $C$ at most once. In that case, the restriction of $\ell$ to $C$ is proper. If $\ell$ starts inside $C$, it may leave and enter again, which makes the restriction improper.
	
	Now assume $I$ can be solved by a single line $\ell_I$, and $J$ by a single line $\ell_J$. Because $I$ is nice, $\ell_I$ ends in its antennae. We obtain a line that solves $I\times J$, by replacing every edge $e$ on $\ell_J$, where $\femin=1$, by a copy of $\ell_I$.
	
	If we instead assume that $I$ requires at least $k$ lines to be solved and $J$ requires at least two, then it follows that $I\times J$ requires at least $k+1$. We prove this by contradiction: Assume there are $k$ lines $\ell_1, \dots, \ell_k$ that solve $I\times J$. 
	Let $\iota\in \{1,2\}$. Because of the special positioning of $A^\iota_{I\times J}$, it can be observed that any line starting inside $A^\iota_{I\times J}$, which leaves $A^\iota_{I\times J}$, cannot enter $A^\iota_{I\times J}$ again. Hence, any line on $I \times J$ can be restricted properly to $A^\iota_{I\times J}$.
	Since each copy of $I$ inside $I\times J$ requires $k$ lines to be solved, we deduce that each line $\ell_i$ for $i\in [1,k]$ intersects both $A^1_{I\times J}$ and $A^2_{I\times J}$. If some line would start inside a copy of $I$, which is neither $A^1_{I\times J}$ nor $A^2_{I\times J}$, then it could not visit both $A^1_{I\times J}$ and $A^2_{I\times J}$, as it would get stuck in an antenna of $I\times J$. From this we conclude that any line $\ell_i$ can be properly restricted to any copy of $I$ in $I\times J$. Again, each copy of $I$ inside $I\times J$ requires $k$ lines to be solved, so now every line $\ell_i$ has to intersect every copy of $I$.
	Choose an arbitrary $\ell_i$. Because it visits each copy of $I$ associated to every edge $e$ of $J$ where $\femin=1$, we can restrict it to the vertices of $J$, and obtain a path $r$ on $J$, visiting all edges $e$ with $\femin=1$. But then $r$ would solve $J$, which contradicts our assumption.
	\begin{figure}
		\centering
		\begin{tikzpicture}[scale=1.8]
			\tikzset{
				k/.style={fill,circle,inner sep=1.0pt},
				e/.style={thick,dashed},
				eh/.style={thick,red}
			}
			\begin{scope}[shift={(0,1)}]
				\node[k] (1a) at (0,0) {};
				\node[k] (2a) at (1,0) {};
				\node[k] (3a) at (0.5,1) {};
				\node[k] (xa) at (0-0.25,0.5) {};
				\node[k] (ya) at (1+0.25,0.5) {};
				
				\node at (0.5+1, 0.5) {$\times$};
				
				\node[k] (1b) at (0+2,0) {};
				\node[k] (2b) at (1+2,0) {};
				\node[k] (3b) at (0.5+2,1) {};
				\node[k] (xb) at (0+2-0.25,0.5) {};
				\node[k] (yb) at (1+2+0.25,0.5) {};
				
				\node at (0.5+3, 0.5) {$=$};
			\end{scope}
			
			\draw[e] (1a) -- (2a);
			\draw[eh] (2a) -- (3a);
			\draw[eh] (3a) -- (1a);
			\draw[eh] (1a) -- (xa);
			\draw[eh] (2a) -- (ya);
			
			\draw[eh] (1b) -- (2b);
			\draw[e] (2b) -- (3b);
			\draw[e] (3b) -- (1b);
			\draw[eh] (1b) -- (xb);
			\draw[eh] (2b) -- (yb);
			
			\begin{scope}[shift={(4.5,0.5)}]
				\node[k] (1B) at (0,0) {};
				\node[k] (2B) at (2,0) {};
				\node[k] (3B) at (1,2) {};
				\node[k] (xB) at (-0.5,1) {};
				\node[k] (yB) at (2.5,1) {};
				
				\draw[e] (2B) -- (3B);
				\draw[e] (3B) -- (1B);
				
				\begin{scope}[shift={(-0.5*0.5+1,-0.5*0.5+0)},scale=0.5]
					\node[k] (1a1) at (0,0) {};
					\node[k] (2a1) at (1,0) {};
					\node[k] (3a1) at (0.5,1) {};
				\end{scope}
				\begin{scope}[shift={(-0.5*0.5-0.25,-0.5*0.5+0.5)},scale=0.5]
					\node[k] (1a2) at (0,0) {};
					\node[k] (2a2) at (1,0) {};
					\node[k] (3a2) at (0.5,1) {};
				\end{scope}
				\begin{scope}[shift={(-0.5*0.5+2.25,-0.5*0.5+0.5)},scale=0.5]
					\node[k] (1a3) at (0,0) {};
					\node[k] (2a3) at (1,0) {};
					\node[k] (3a3) at (0.5,1) {};
				\end{scope}
				
				\draw[e] (1a1) -- (2a1);
				\draw[eh] (2a1) -- (3a1);
				\draw[eh] (3a1) -- (1a1);
				\draw[eh] (1a1) -- (1B);
				\draw[eh] (2a1) -- (2B);
				\draw[e] (1a2) -- (2a2);
				\draw[eh] (2a2) -- (3a2);
				\draw[eh] (3a2) -- (1a2);
				\draw[eh] (1a2) -- (xB);
				\draw[eh] (2a2) -- (1B);
				\draw[e] (1a3) -- (2a3);
				\draw[eh] (2a3) -- (3a3);
				\draw[eh] (3a3) -- (1a3);
				\draw[eh] (1a3) -- (2B);
				\draw[eh] (2a3) -- (yB);
				
			\end{scope}

		\end{tikzpicture}
		\caption{Example construction of $I \times J$. Edges with $\fmin_e=1$ are red, other edges are dashed.}\label{fig:ItimesJ}
	\end{figure}
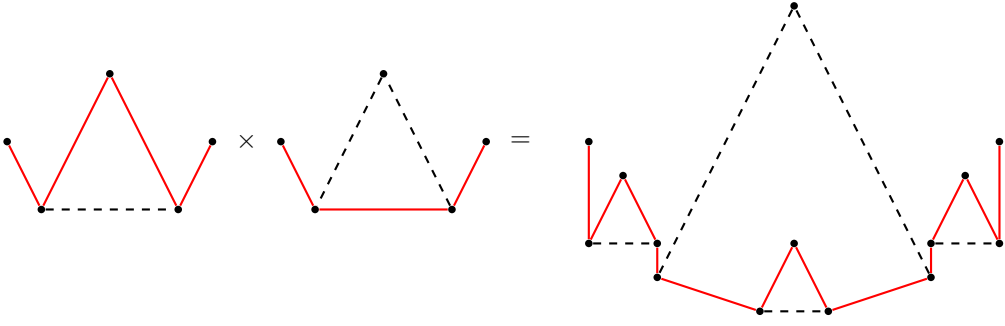
	
	For an \LPAL{} instance $I$ and a number $k\in \mathbb{N}_{\geq 1}$ we define $I^k$ as the repeated product $((I \times I) \times ...) \times I$ of $k$ factors.
	If $I$ can be solved using a single line, $I^k$ can as well. Otherwise $I^k$ requires at least $k+1$ lines. 
	$I^k$ contains at most $n^{2k}$ vertices.
	
	Now assume that for some $\alpha \in [1,\infty)$, \LPAL{} can be approximated within $\alpha$ using an algorithm $A$ in polynomial time. Define $k:= \lfloor \alpha \rfloor$. We show how to decide the Hamiltonian path problem in polynomial time, implying $P=NP$. 
	
	Let $G$ be a directed graph. Apply the reduction from \cite{GHS16} to $G$ to obtain an \LPAL{} instance $I_0$ that is solvable using a single line if and only if $G$ has a Hamiltonian path. If $I_0$ has more than two antennae, we know $G$ has no Hamiltonian path. If $I_0$ has two or fewer antennae, we consider all possible ways to attach the missing antennae (if $I_0$ is nice, there is only one way, i.e.\ attaching none). This gives us a list $L$ of at most $n^2$ nice instances. If $I_0$ is solvable using a single line, then some $I \in L$ is too. We repeat the following for every $I\in L$:
	
	First construct $I^k$. This is possible in polynomial time, since $k$ does not depend on $n$.
	Apply $A$ to $I^k$ to obtain an approximately optimal line concept that has cost $x$. 
	If $I$ can be solved using one path, the minimal cost of solving $I^k$ is 1. Hence $x \leq \alpha$.
	Otherwise the minimal cost of solving $I^k$ is $k+1$, hence $x\geq k+1>\alpha$.
	It follows that by comparing $x$ to $\alpha$, we can determine whether $I$ can be solved using one path.
	
	If none of the $I\in L$ can be solved using one path, then $G$ has no Hamiltonian path. Otherwise we know $G$ has a Hamiltonian path.
\end{proof}

As this hardness result is weaker, we could hope to find an approximation algorithm where the error grows only very slightly in $n$. This is an interesting open problem.

\section{Optimal line planning for stars}\label{sec:poly}

While \LPAL{} is NP-hard for paths if $\dfix>0$, the problem is easier when no frequency-independent costs are considered, i.e., for $\dfix=0$. Here, the costs do not increase if edges are covered by  multiple lines, ending at different terminals. We can show that optimal solutions have a special structure by rewriting the cost function
\begin{equation}
\cost((\cL,f)) = \underbrace{\dfix}_{=0} \cdot |\cL| + \sum_{\ell \in \cL} \cost_{\ell} \cdot  f_{\ell} 
= \sum_{e\in E}c_e \cdot \ftotal_e + \cfix \cdot \sum_{\ell \in \cL} f_l . \label{eq:cost_star}
\end{equation} 

As all edges in a star are incident to a central vertex, there is an optimal solution where each edge $e \in E$ is covered exactly $\femin$ times, i.e., $\ftotal_e=\femin$. % in an optimal solution. 
Thus, it remains only to minimize the frequency-dependent fixed costs $\cfix\cdot \sum_{\ell \in \cL} f_\ell$ in \eqref{eq:cost_star}. As each line contains either one or two edges and two-edge lines reduce the costs by $\cfix$, this is equivalent to minimizing the total frequency of one-edge lines.

It is easy to see that each of the following conditions guarantees optimality of the line concept as in each case as many edges as possible are ``paired up'' to two-edge lines:
\begin{enumerate}
	\item There is no one-edge line.
	\item There is one one-edge line with frequency one.
	\item There is an edge with $\bar{e} \in E$ with $\fmin_{\bar{e}}> \sum_{e \in E\setminus \{\bar{e}\}} \femin$ and 
					%$ \cfix\cdot \sum_{\ell \in \cL}f_\ell =\cfix \cdot \fmin_{\bar{e}}.$
					$\sum_{\ell \in \cL}f_\ell =\fmin_{\bar{e}}.$
\end{enumerate}

\begin{algorithm}[h]
	\caption{Finding an optimal solution to \LPAL{} for stars}\label{alg-star} %with $\dfix=0$, $c_e > 0$ $\forall e \in E$, $\cfix>0$}\label{alg-star}
	\textbf{Input:} An instance $(G, \dfix, \cfix, c, \fmin, \fmax)$ where $\dfix=0$ and $G=(V,E)$ is a star.
	\begin{algorithmic}[1]
		\State $\elist =[e_1, \ldots, e_m]$ list of edges in $E$, sorted decreasingly by $\femin$, $\bar{E} = \emptyset$, $\fmint=\fmin$
		\State $f_{(e)}=0$, $f_{(e_i,e_j)}=0$ for all $e,e_i,e_j \in E$, $i>j$ 
		\For{$e_k \in \elist$} \label{alg:loop1}
		%\If{there is $\ell$ with $f(\ell) >0$, $E(\ell)=\{\bar{e}\}$}
		\If{there is $\bar{e} \in \bar{E}$}				
		\State $a = \min \{f_{(\bar{e})}, \fmin_{e_k}\}$
		\State $f_{(\bar{e})} \mathrel{{-}{=}} a$, $f_{(e_k, \bar{e})} = a$
		\State $\fmint_{e_k} \mathrel{{-}{=}} a$	
		\If{$f_{(\bar{e})}=0$}
		\State $\bar{E}=\emptyset$
		\EndIf
		\EndIf
		%\For{$\ell$ with $E(\ell)=\{e_1,e_2\}$, $e \notin E(\ell)$}
		\For{$e_i,e_j \in \{e_1, \ldots, e_{k-1}\}$ with $i>j$, $f_{(e_i,e_j)}>0$ and $\fmint_{e_k} >1$}			
		%\If{$\fmint_{e_k} >1$}
		\State $b = \min \left\{ f_{(e_i,e_j)}, \Bigl\lfloor \frac{\bar{f}_{e_k}^{\min}}{2}\Bigr\rfloor \right\}$
		
		\State $f_{(e_i,e_j)} \mathrel{{-}{=}} b$, $f_{(e_k, e_i)} \mathrel{{+}{=}} b$, $f_{(e_k, e_j)} \mathrel{{+}{=}} b$
		\State $\bar{f}_{e_k}^{\min} \mathrel{{-}{=}} 2\cdot b$	
		%\Else
		%	\State break
		%\EndIf
		\EndFor	
		\If{$\fmint_{e_k} >0$}
		\State $f_{(e_k)} = \fmint_{e_k}$, $\bar{E}=\{e_k\}$
		\EndIf		
		\EndFor
		\State $\cL =\{(e_i,e_j) \colon f_{(e_i,e_j)}>0\} \cup \{(e)\colon f_{(e)}>0\}$, $f = f|_\cL$
		\State Output $(\cL,f)$
	\end{algorithmic}
\end{algorithm}

In \autoref{alg-star}, we present a polynomial time algorithm that finds an optimal solution to \LPAL. % satisfying one of the conditions above. 
Starting with a list of edges sorted by decreasing $\femin$, \LPAL{} is iteratively solved for the first $k$ edges, $k \in \{1, \ldots, |E|\}$ such that one of the optimality conditions 1, 2 or 3 is satisfied at the end of each iteration. 
The one-edge lines with positive frequency are stored in the set $\bar{E}$ which never contains more than one edge. 

After iteration 1, $\bar{E} = \{e_1\}$ and condition~3 is satisfied.
In iteration $k$, the edge $e_k$ is paired up with edge $\bar{e} \in \bar{E}$ creating a new two-edge line if $\bar{E}$ is not empty. If $f_{(\bar{e})} > \fmin_{e_k}$, $\fmint_{e_k}$ is reduced to zero, $\bar{E} = \{\bar{e}\}$ and condition~3 is satisfied. If $f_{(\bar{e})} = \fmin_{e_k}$, $\fmint_{e_k}$ and $f_{(\bar{e})}$ are reduced to zero, $\bar{E} = \emptyset$ and condition~1 is satisfied. %Otherwise, 
If $f_{(\bar{e})} < \fmin_{e_k}$or $\bar{E} = \emptyset$ in line~4, $\bar{E} = \{e_k\}$ in the for-loop starting in line~12 and we have to show that at the end of the iteration either condition~1 or 2 is satisfied. As the list of edges is sorted by decreasing $\femin$, we know that the total frequency of all already constructed lines is at least $\frac{\fmin_{e_k}}{2}$ such that we can split already existing lines and create two new ones containing $e_k$. Thus, in line~17 $\fmint_{e_k}$ is either zero or one, such that optimality condition~1 or 2 is satisfied and we get the following theorem.  

\begin{theorem}\label{thm:stars_poly}
	\autoref{alg-star} finds an optimal solution to \LPAL{} for stars with $\dfix=0$ in $\mathcal{O}(n^3)$. % polynomial time.
\end{theorem}
\begin{proof}
	Note that \autoref{alg-star} computes a line concept that covers each edge $e \in E$ exactly $\femin$ times, i.e.\ $\ftotal_e =\femin$. To prove optimality for $\dfix=0$, we therefore only have to show that the total frequency of one-edge lines is minimized.
	
	At the start of each for loop in line~\ref{alg:loop1}, $\bar{E}$ contains the edges for which a one-edge line with positive frequency exists. Note that there is always at most one edge $\bar{e} \in \bar{E}$, as by the choice of $a$ in line~5, $\femin$ can only be positive if $f_{(\bar{e})}$ is set to zero. Thus the line concept $(\cL,f)$ created in line~22 contains at most one one-edge line with positive frequency. %As $\elist$ is sorted according $\fmin$, for all but the first edge $\femin$ is either zero or one in line~21, such that only the first edge can 
	\begin{itemize}
		\item If there is no one-edge line, the line concept is optimal as in condition~1. 
		\item If there is a one-edge line containing the first edge $e_1$ of $\elist$, i.e.\ the edge with the highest $\fmin$, then in line~5 the minimum $a$ is always chosen as $\femin$ for $e \neq e_1$, i.e.\ $\fmin_{e_1} > \sum_{e \neq e_1} \femin$. % such that $(\cL,f)$ is optimal.
		In this case, all lines contain edge $e_1$ and thus $\sum_{\ell \in \cL} f_\ell = \fmin_{e_1}$ such that $(\cL,f)$ is optimal, see condition~3.
		\item  Otherwise, there is a one-edge line that does not contain the first edge. Here, we show that for any $e_k \neq e_1$ in the outer-loop (lines~3 to 20) with  $\fmint_{e_k}>0$ in line~17 also  $\fmint_{e_k}=1$ holds. Then, we have one one-edge line with frequency one and the line concept is optimal according to condition~2.
		
		As $\fmint_{e_k}$ is only reduced in the algorithm, $\fmint_{e_k}>0$ can only hold in line~17 if it already holds before the for-loop starting in line~12.  Note that in this case, $k >2$ holds. We want to show that $\fmint_{e_k}$ is reduced in the for-loop (lines~12 to 16) until $\fmint_{e_k}\in \{0,1\}$. Suppose to the contrary, that $\fmint_{e_k} > 1$ in line~17. Then the minimum $b$ chosen in line~13 always has been chosen as $f_{(e_i,e_j)}$ and we get
		\[\sum_{\substack{(e_i,e_j):\\i,j <k}} f_{(e_i,e_j)} < \alpha\]
		where $\alpha$ is the value of $\fmint_{e_k}$ before starting the for-loop in line~12. We know that $\alpha = \fmin_{e_k}$ if $\bar{E} = \emptyset$ in line~4 and $\alpha =  \fmin_{e_k} - f_{(e_k, \bar{e})}$ if $\bar{E} = \{\bar{e}\}$ in line~4. To simplify the notation in the following we set $f_{(e_k,\bar{e})}=0$ if $\bar{E}=\emptyset$. As at the beginning of the for-loop in line~4 for $e_k$ all edges $e_i$, $i \in \{1, \ldots, k-1\}$, are covered $\fmin_{e_i}$-times we get
		\[\sum_{\substack{(e_i,e_j):\\i,j <k}} f_{(e_i,e_j)} + f_{(e_k,\bar{e})} \geq  \frac{1}{2}\sum_{i=1}^{k-1}\fmin_{e_i} \geq \frac{1}{2} \cdot 2\cdot \fmin_{e_k} = \fmin_{e_k}\]
		and thus 
		\[\sum_{\substack{(e_i,e_j):\\i,j <k}} f_{(e_i,e_j)} \geq  \fmin_{e_k} -f_{(e_k,\bar{e})} = \alpha\]
		which is the desired contradiction.
	\end{itemize}

	The runtime of \autoref{alg-star} can be 	estimated in the following way: There are $|E|=|V|-1=n-1$ iterations of the outer for-loop starting in line~3 and $\mathcal{O}(n^2)$ iterations of the inner for-loop starting in line~12. As sorting $\elist$ in line~1, initializing the frequencies in line~2 and reconstructing the line concept in line~22 are also in $\mathcal{O}(n^3)$, the total runtime of \autoref{alg-star} is$\mathcal{O}(n^3)$. 
\end{proof}

\section{Optimal line planning for trees}\label{sec:trees}
Since paths are special instances of trees, \LPAL{} is NP-hard on trees by \autoref{thm:dhard_path}.
If we assume that $\dfix = 0$ and that $\fmax$ is bounded by a constant~$b$, then we can provide a pseudo-linear time algorithm for finding the optimal objective value of \LPAL{} on trees.
\begin{theorem} \label{thm:trees_linear}
	If $T$ is a tree, $\dfix = 0$, and $\fmax$ is bounded by a constant $b$, then the minimal cost for \LPAL{} can be computed in $\mathcal{O}(nb^3)$. An optimal line concept can be computed in~$\mathcal{O}(n^3b^3)$.
\end{theorem}

\subparagraph{Intuition.}
It is well known that a rooted tree $(T,r)$ can be constructed from the set of its leafs by starting with the set 
$\{((\{v\}, \emptyset), v) \colon \text{$v$ is a leaf of $T$}\}$ of rooted singleton trees for all leafs of $T$ and iteratively introducing parents and merging subtrees. For our dynamic program it is crucial that we restrict these operations further. We modify the set of rooted subtrees by the following two operations:
\begin{itemize}
	\item \emph{introduce a parent:} a subtree $(T',r')$ can be extended by a parent $p \in V(T) \setminus V(T')$ if $p$ is the only neighbor of $r'$ that is not contained in $V(T')$,  
	\item \emph{merge:} two subtrees $(T_1,r')$, $(T_2,r')$ can be merged at the same root $r'$ if $r'$ has only one child in one of the trees $T_1$ or $T_2$.
\end{itemize}
When no further operation of these types can be applied anymore, the set of subtrees only contains $(T,r)$ as desired.
   
We exploit that there exists an optimal solution for \LPAL{} with the following property: The restriction of this optimal solution to a rooted tree $(T',r')$ arising in the above construction satisfies that at most $b$ lines end in $r'$ (otherwise a merge of two such lines would give a solution of lower costs).
We compute the optimal value for \LPAL{} using the above construction where
each subtree has a table which stores its optimal solutions, considering any possible number of lines ending in its root.
If $(\mathcal{L},f)$ is a line concept for $T$, then for each $v \in V(T)$ we define \emph{the number of lines ending at $v$} as \[\lineends{v}((\mathcal{L},f)) \coloneqq \sum_{\substack{\ell \in \mathcal{L} \colon \text{$v$ is} \\ \text{an end of $\ell$}}}f_{\ell}\]
where we allow zero-edge lines.
The cost of an optimal solution satisfying 
$\lineends{v} \geq k$ is
\[\cost(T \mid \lineends{v} \geq k) \coloneqq \min\{\cost((\mathcal{L},f)) \mid (\mathcal{L},f)\in \feasible(T),\; \eta_v((\mathcal{L},f)) \geq k\}. \]
We compute the \emph{cost vector} 
\[\cost(T',r') := (\cost(T' \mid \lineends{r'} \geq 0), \cost(T' \mid \lineends{r'} \geq 1),  \dots, \cost(T' \mid \lineends{r'} \geq b))\] 
for each rooted subtree $(T',r')$ appearing in the above construction. The recursive computation  stores intermediate results in a table to avoid re-computation.
Finally, the cost of an optimal line concept for $T$ is $\cost(T \mid \lineends{r} \geq 0)$.

\begin{lemma} \label{lem: dynamic-program-tree}
	Let $(T',r')$ be a rooted tree and $k \in \{1, \dots, b\}$.
	\begin{enumerate}
		\item \label{itm: trivial-tree} If $|V(T')| = 1$, then
		$\cost(T' \mid \lineends{r'} \geq k) = k \cdot \cfix$.
				The time required to compute \linebreak $\cost(T' \mid \lineends{r'} \geq k)$ is $\mathcal{O}(1)$ and, hence, the time required to compute  $\cost(T',r')$ is $\mathcal{O}(b)$.
		\item \label{itm: one-child} If $\deg_{T'}(r') = 1$ and $u$ denotes the child of $r'$ in $T'$, then $\cost(T' \mid \lineends{r'} \geq k)$ equals
		\[\min_{0 \leq m \leq \max\{k, \fmin_{\{u,r'\}}\}}\{ \cost(T'-r' \mid \lineends{u} \geq m) + \max\left\{k, \fmin_{\{u,r'\}}\right\} \cdot (\cfix+c_{\{u,r'\}}) - m \cdot \cfix  \}.\]
		If the values $\cost(T'-r' \mid \lineends{u} \geq m)$ are pre-computed for all $m \in \{1, \dots, b\}$, then
		the time required to compute $\cost(T' \mid \lineends{r'} \geq k)$ is $\mathcal{O}(b)$ and, hence, $\cost(T',r')$ can be computed in $\mathcal{O}(b^2)$ time.
		\item \label{itm: merge} If $(T',r')$ is the union of two rooted trees $(T_1,r'), (T_2,r')$ where $\deg_{T_1}(r') = 1$, then $\cost(T' \mid \lineends{r'} \geq k)$ equals
		\[\min_{\substack{0 \leq m, k_1, k_2 \leq b, \\ k_1+k_2-2m = k}}\{\cost(T_1\mid \lineends{r'} \geq k_1) + \cost(T_2\mid \lineends{r'} \geq k_2) - m \cdot \cfix \}.\]
		If the values $\cost(T_2\mid \lineends{r'} \geq k_2)$ and $\cost(T_2\mid \lineends{r'} \geq k_2)$ are pre-computed for all $k_1, k_2 \in \{1, \dots, b\}$, then
		the time required to compute $\cost(T' \mid \lineends{r'} \geq k)$ is $\mathcal{O}(b^2)$ and, hence it requires $\mathcal{O}(b^3)$ time to compute $\cost(T',r')$.
	\end{enumerate}
\end{lemma}
\begin{proof}[Proof of \autoref{lem: dynamic-program-tree}] \label{prf:lem: dynamic-program-tree}
	If $(T', r')$ has only one vertex, then clearly the optimal line concept which satisfies that $k$ lines end in $r'$ consists of $k$ zero-edge lines. This implies~\eqref{itm: trivial-tree}.
	
	We prove~\eqref{itm: one-child}.
	Since $\deg_{T'}(r') = 1$ every line $\ell$ in $T'$ is either contained in $T'-r'$ or it has one end in $T'-r'$ and the other end is $r'$.
	In a line concept $(\mathcal{L},f)$ of $T'$, a line $\ell$ with one end in $T'-r'$, the other end being $r'$ and frequency $f_\ell$ can be split into two lines $\ell_1=(r',u)$ and $\ell_2=\ell - r'$ with frequency $f_{\ell}$ without changing the feasibility.
	The line $\ell_2$ is contained in $T'-r'$ and the cost of the line concept is increased by $\cfix \cdot f_\ell$.
	This process can be reversed, merging some line from $T'-r'$ that ends at $u$ with the line $(u,r')$, decreasing the cost accordingly.
	Assuming $k\leq\fmax_{\{u,r'\}}$, this allows us to rewrite $\cost(T' \mid \lineends{r'} \geq k)$: 
	\begin{align*}
		&\cost(T' \mid \lineends{r'} \geq k) = \min\{\cost((\mathcal{L},f)) \colon (\mathcal{L},f)\in \feasible(T'),\; \eta_{r'}((\mathcal{L},f)) \geq k\} \\
		&\overset{(a)}{=} \min\{\cost((\mathcal{L'},f')) + a \cdot (\cfix+c_{\{u,r'\}}) - m \cdot \cfix \colon (\mathcal{L'},f')\in \feasible(T' - r'),\\&\hspace{6em} \fmin_{\{u,r'\}} \leq a \leq \fmax_{\{u,r'\}},  m \leq \eta_u((\mathcal{L'},f')), m \leq a, a \geq k \} \\
		&\overset{(b)}{=} \min\{\cost((\mathcal{L'},f')) + \max\{k, \fmin_{\{u,r'\}}\} \cdot (\cfix+c_{\{u,r'\}}) - m \cdot \cfix \colon (\mathcal{L'},f')\in \feasible(T' - r'), \\&\hspace{6em} \max\{k, \fmin_{\{u,r'\}}\} \leq \fmax_{\{u,r'\}} , m \leq \eta_u((\mathcal{L'},f')), m \leq \max\{k, \fmin_{\{u,r'\}}\} \} \\
		&\overset{(c)}{=} \min_{0 \leq m \leq \max\{k, \fmin_{\{u,r'\}}\}}\{ \cost(T'-r' \mid \lineends{u} \geq m) + \max\{k, \fmin_{\{u,r'\}}\} \cdot (\cfix+c_{\{u,r'\}}) - m \cdot \cfix  \}
	\end{align*}
	(a): We split the lines in $T'$ into some set of lines $\cL'$ on $T'-r'$, and $a$ copies of the line $(u,r')$, from which $m$ are merged with lines from $\cL'$. Then the number of ends at $r'$ is exactly $a$, hence it satisfies $a \geq k$. Furthermore $a \in [\fmin_{\{u,r'\}}, \fmax_{\{u,r'\}}]$. Each merge reduces the cost by $\cfix$.\\
	(b): To minimize the cost, we have to minimize $a$: the only benefit of increasing $a$ is that $m$ can be increased but the factor of $a$ outweighs $m$. Hence we replace $a$ by its minimum possible value $\max\{k, \fmin_e\}$.  \\
	(c): Since $m \leq \eta_{u}((\mathcal{L'},f'))$ we can replace $\cost((\mathcal{L'},f'))$ by $\cost(T'-r' \mid {\lineends{u} \geq m})$. The condition $\max\{k, \fmin_{\{u,r'\}}\} \leq \fmax_{\{u,r'\}}$ is fulfilled by the assumption on $k$. The remaining constraints are written as a subscript.\\
	If $k > \fmax_{\{u,r'\}}$, then $\cost(T' \mid \lineends{r'} \geq k) = \infty$ since no feasible line concept with $\lineends{r'} \geq k$ exists.
	
	The time to compute $\cost(T' \mid \lineends{r'} \geq k)$ for some $k$ is $\mathcal{O}(b)$, since $\max\{k, \fmin_{\{u,r'\}}\} \leq b$. Hence $\cost(T',r')$ can be computed in $\mathcal{O}(b^2)$.
	
	Finally, we prove~\eqref{itm: merge}. Any line in $T'$ that traverses $r'$ can be split into two lines, one contained in $T_1$ and the other contained in $T_2$. In reverse, we can join lines from different subtrees together at $r'$. Then
	\begin{align*}
		&\cost(T' \mid \lineends{r'} \geq k) = \min_{\substack{0 \leq m, k_1, k_2 \leq b, \\ k_1+k_2-2m = k}}\{\cost(T_1\mid \lineends{r'} \geq k_1) + \cost(T_2\mid \lineends{r'} \geq k_2) - m \cdot \cfix \}
	\end{align*}
	Note that at most $b$ lines of $T_1$ end at $r'$ by the degree condition.
	The time required to compute $\cost(T' \mid \lineends{r'} \geq k)$ for some $k$ is $\mathcal{O}(b^2)$ since we have two degrees of freedom in the minimum expression. Hence $\cost(T',r')$ can be computed in $\mathcal{O}(b^3)$.
\end{proof}

\subparagraph{Total runtime.}
A depth-first search algorithm yields a decomposition of $T$ such that the dynamic programming approach can be executed in the corresponding order. Since $T$ is a tree, DFS has a running time of $\mathcal{O}(n)$.
The running time to compute the cost vector for all leaves in the initial set $S$ is in $\mathcal{O}(nb)$ since there are at most $n-1$ leaves in $T$ and by \ref{lem: dynamic-program-tree}.\eqref{itm: trivial-tree}.
We need to introduce a parent in the construction of $T$ exactly $|E(T)| = n-1$ times. Together with \ref{lem: dynamic-program-tree}.\eqref{itm: one-child} this yields that computing the respective cost vectors has a total running time of $\mathcal{O}(nb^2)$.
Finally, the merge operation is performed $\mathcal{O}(\sum_{v \in V(T)} \deg_T(v)) = \mathcal{O}(n)$ times which gives a total running time of $\mathcal{O}(nb^3)$.
Altogether, the dynamic programming has a running time of $\mathcal{O}(nb^3)$.

\subparagraph{Constructing a line concept.}
So far we only showed how to compute the minimal cost among all feasible line concepts. 
To actually construct a line concept with minimal cost, we store in each cost vector entry additionally a line concept of that cost. These line concepts can then also be computed recursively, according to the decisions made (i.e.\ creating zero-edge line, extending lines by a single edge, joining lines). This increases the algorithm runtime, depending on the line concept sizes. On a tree, we can have at most $\mathcal{O}(n^2)$ different paths. It is then possible to compute all cost vectors augmented with line concepts in time $\mathcal{O}(n^3 b^3)$.

Altogether, this proves \autoref{thm:trees_linear}.

Since the runtime of the algorithm depends on $b$, it is \emph{pseudo-polynomial}. For the special case where for all $e\in E$ it holds $\femin=\femax$, we provide a true polynomial time algorithm, which does not depend on a frequency bound $b$.
\begin{algorithm}[h]
\caption{Finding an optimal solution of \LPAL{}  on trees with $\fmin=\fmax$}\label{alg-trees-fixed-frequency} 
	\textbf{Input:} An instance $(G, \dfix, \cfix, c, \fmin, \fmax)$ where $\dfix=0$, $\fmin=\fmax$ and $G=(V,E)$ is a tree.
\begin{algorithmic}[1]
	%\State Initial feasible line concept: $(\cL, f)$ with
	\State $\cL = \{(e) \colon e \in E \}$
	\State $f_{(e)} = \fmin_e$ for all $e \in E$; for all other paths $\ell$ set $f_\ell=0$
	\For{$v \in V$} \label{alg2:for1}
		\State Let $S$ be the star formed by $v$ and its neighbors.
		\State Let $(\cL^S, f^S)$ be the result of \autoref{alg-star} applied to the sub-instance on $S$.
		\State $L_v = \{ \ell \in \cL \colon \ell \text{ ends in }v \}$
		\For{$\ell_1,\ell_2 \in L_v$} \label{alg2:for2}
			\State Let $e_1$ be the edge of $\ell_1$ incident to $v$
			\State Let $e_2$ be the edge of $\ell_2$ incident to $v$
			\If{$e_1 = e_2$}
				\State \textbf{continue}
			\EndIf
			\State $d = \min\{ f^S_{(e_1,e_2)}, f_{\ell_1}, f_{\ell_2} \}$
			\State $\ell_{+} = \ell_1 \cup \ell_2$
			\State $\cL = \cL \cup \{ \ell_{+} \}$
			\State $f^S_{(e_1,e_2)} \mathrel{{-}{=}} d$, $f_{\ell_1} \mathrel{{-}{=}} d$, $f_{\ell_2} \mathrel{{-}{=}} d$, $f_{\ell_{+}} \mathrel{{+}{=}} d$
		\EndFor
	\EndFor
	\State $\cL=\{\ell \in \cL \colon f_{\ell} > 0 \}$, $f = f|_\cL$
	\State Output $(\cL, f)$
\end{algorithmic}
\end{algorithm}
\begin{theorem} \label{thm:trees_fmin=fmax}
	If $G$ is a tree, $\dfix = 0$, and $\femin=\femax$ for all $e\in E$, then \autoref{alg-trees-fixed-frequency} computes an optimal solution to \LPAL{} in $\mathcal{O}(n^3)$.
\end{theorem}
\begin{proof}
	The key idea of \autoref{alg-trees-fixed-frequency} is to apply \autoref{alg-star} iteratively at every vertex. As $\fmin=\fmax$, we can handle lines ending at vertex $v \in V$ in the same way we handle edges in stars: creating a two-edge line in a star corresponds to concatenating two lines in a tree.  
	
	We show that \autoref{alg-trees-fixed-frequency} computes a feasible and optimal solution; then we compute its runtime. 
	
	After line 2, a feasible line concept is constructed. The operations in line 16 simply merge lines, hence the feasibility of $(\cL, f)$ remains.
	
	For showing optimality, we first note that since the total frequencies $\ftotal_e$ are fixed for every $e \in E$, obtaining an optimal line concept $(\cL, f)$ is equivalent to minimizing $\sum_{\ell \in \cL} f_{\ell}$. Since every line has two ends, another equivalent quantity to minimize is the total number of line ends, weighted by $f$, i.e.\ $2\sum_{\ell \in \cL} f_{\ell}$.
	
	Define $L_{v,e} := \{ \ell \in \cL \colon \ell \text{ ends in }v \text{ and traverses }e \}$.
	We need an invariant (I1) that holds before every iteration of the outer for-loop: For every vertex $v\in V$ that has not yet been chosen in the outer for-loop, we have $\femin = \sum_{\ell \in L_{v,e}} f_{\ell}$.
	Clearly (I1) holds directly after executing line 2. The operations during an iteration only affect the local line ends, i.e.\ the number of ends at yet unvisited vertices is unchanged. Hence (I1) is maintained.
	
	Another invariant (I2), that holds before every iteration of the inner loop, for every $e$ incident to $v$, is $\sum_{\ell \in L_{v,e}} f_{\ell} = f^S_{(e)} + \sum_{e' \neq e} f^S_{(e,e')} $. 
	It holds initially, since \autoref{alg-star} produces a feasible line concept, and we have $\femin = f^S_{(e)} + \sum_{e' \neq e} f^S_{(e,e')}$; combine this with (I1) to obtain (I2). Let $e_1$ and $e_2$ be chosen during an iteration, after line 9. The operations inside the loop only affect lines that contain $e_1$ or $e_2$, hence for any $e\notin\{ e_1, e_2\}$ (I2) is maintained. (I2) is also maintained for $e_1$, since $f^S_{(e_1,e_2)}$ and $f_{\ell_1}$ are changed by equal amounts. The same holds true for $e_2$.
	
	We claim that after the inner for-loop finishes, we have $f^S_\ell=0$ for all two-edge lines $\ell=(e_1,e_2)$ of $\cL^S$. This is proved by contradiction: Assume $f^S_{(e_1,e_2)}>0$ for some $e_1 \neq e_2$. 
	Then by (I2), $\sum_{\ell \in L_{v,e_1}} f_{\ell} = f^S_{(e_1)} + \sum_{e' \neq e_1} f^S_{(e_1,e')} > 0$, and similarly $\sum_{\ell \in L_{v,e_2}} f_{\ell} > 0$. Hence two lines $\ell_1\in L_{v,e_1}$ and $\ell_2\in L_{v,e_2}$ exist with $f_{\ell_1}>0$ and $f_{\ell_2}>0$. But then the inner for-loop would have chosen $\ell_1$ and $\ell_2$ at some point, after which $f_{\ell_1}$, $f_{\ell_2}$ or $f^S_{(e_1,e_2)}$ would have been zero, which is the desired contradiction.
	
	Using (I2) again, we have $\sum_{\ell \in L_{v,e}} f_{\ell} = f^S_{(e)}$ after the inner for-loop. This means that we have $\sum_{e \text{ incident to } v} f^S_{(e)} =: x_v$ line ends, with multiplicity, at vertex $v$. Because of the algorithm's locality, this number does not change in further iterations of the outer loop.
	
	In total \autoref{alg-trees-fixed-frequency} produces a line concept with $\sum_{v\in V} x_v$ line ends. Now assume there exists a better solution, i.e.\ a feasible line concept $(\cL', f')$ that has fewer than $x_v$ line ends at some vertex $v$. Then we could restrict $(\cL', f')$ onto the star $S$ around $v$ and would obtain a solution for $S$ which has fewer ends, i.e.\ is better, than what \autoref{alg-star} computed, which contradicts the optimality of \autoref{alg-star}.

	On the runtime: To speed up operations, lines are represented just by their end vertices. Since we are on a tree, this is enough to unambiguously define them. The invocation of \autoref{alg-star} can be done in $\mathcal{O}(\deg(v)^3)$. 
	Since $L_v$ has at most $n$ elements, the for-loop at line \ref{alg2:for2} iterates at most $n^2$ times.
	Every operation inside the for-loop takes constant time and we can bound the total loop runtime by $\mathcal{O}(n^2)$.
	Overall, an iteration of the outer for-loop on a vertex $v$ takes $\mathcal{O}(\deg(v)n^2 )$. Using the fact that on a tree $\sum_{v\in V} \deg(v) = 2n-2$, the total runtime of the algorithm is $\mathcal{O}\left( \sum_{v\in V} \deg(v)^3 + n^2\sum_{v\in V} \deg(v) \right) = \mathcal{O}(n^3)$.
	
\end{proof}

\section{Conclusion and outlook}\label{sec:outlook}
We systematically investigated the complexity of the line planning on all lines problem. Using frequency-independent line costs results in
an
NP-hard problem
even for paths and stars. Without these costs, the problem remains NP-hard on planar graphs but can be solved in polynomial time on trees when $\fmin=\fmax$, and in pseudo-polynomial time otherwise.

\begin{samepage}
\noindent The following are the most pressing open questions:
\begin{itemize}
	\item Is \LPAL{} in NP? It is not clear that, especially when $\fmin$ is very large, the size of an optimal line concept can be bounded by a polynomial in the input size. 
	\item Is there a polynomial time algorithm for \LPAL{} with $\dfix=0$ on trees?
	\item Is there a (pseudo-)polynomial time algorithm for \LPAL{} with $\dfix=0$ on graphs with treewidth 2 (or generally bounded treewidth)?
	\item Under which restrictions exists a constant-factor polynomial-time approximation algorithm for \LPAL{}?
\end{itemize}
\end{samepage}
On graphs of bounded treewidth, many NP-complete problems become easy \cite{DBLP:journals/jal/ArnborgLS91}. \LPAL{} however does not fit into the problem types studied before. We conjecture that a dynamic programming approach, similar to the one we used on trees, can be used on graphs with treewidth 2. However, there are some complications: when combining partial solutions at a larger separator, we need to make sure that merged lines do not form cycles or self-intersections which is by construction not possible in trees.

When moving from trees to graphs of higher treewidth, an additional problem has to be considered: while for trees we can assume that passenger paths are fixed, this is no longer true in general graphs. Thus, looking at line planning from a passenger' perspective, it might be beneficial to replace the lower frequency bounds $\fmin$ by a flow formulation for the passengers as in \cite{borndorfer2007column} such that passengers can choose routes in the network for which the capacity has to be sufficiently high. 
This presents an interesting extension of the problem, where it is especially important to understand the structure of optimal solutions.

\newpage

\bibliography{literature}

\end{document}